\documentclass[a4paper,times,3p]{elsarticle}
\pdfoutput=1
\usepackage[utf8]{inputenc}

\usepackage[T1]{fontenc}

\usepackage{microtype}

\usepackage{booktabs}

\usepackage{amsmath}
\usepackage{amssymb}
\usepackage{amsthm}
\usepackage{bm}
\usepackage[hidelinks]{hyperref}
\usepackage{graphicx}

\usepackage{mathtools}
\newcommand{\aligneqn}[2]{
	\mathmakebox[\widthof{#1}][c]{#2}
}

\usepackage[font=footnotesize,labelfont=bf]{caption}
\usepackage{subcaption}
\usepackage{multirow}

\usepackage{pxfonts}

\usepackage{import}

\usepackage{gnuplot-lua-tikz}

\usepackage{aicescover}

\usepackage[textsize=scriptsize]{todonotes}
\usepackage{setspace}

\newtheorem{thm}{Theorem}
\newtheorem{lem}{Lemma}
\newtheorem{cor}{Corollary}
\theoremstyle{remark}
\newtheorem{rem}{Remark}

\def\statement{\begin{minipage}[t]{.75\textwidth}
       NOTICE: this is the author's version of a work that was accepted for publication in Journal of Non--Newtonian Fluid Mechanics. Changes resulting from the
       publishing process, such as peer review, editing, corrections, structural formatting, and other quality control mechanisms may not be reflected in this
       document. Changes may have been made to this work since it was submitted for publication. A definitive version was subsequently published in Journal of
       Non--Newtonian Fluid Mechanics, [214, (2014)]  \href{http://dx.doi.org/10.1016/j.jnnfm.2014.09.018}{DOI:10.1016/j.jnnfm.2014.09.018}
       \end{minipage}}

\makeatletter
\def\ps@pprintTitle{%
     \let\@oddhead\@empty
     \let\@evenhead\@empty
     \def\@oddfoot{\footnotesize\itshape
       \statement\hfill\today}%
     \let\@evenfoot\@oddfoot}
\makeatother

\journal{Journal of Non-Newtonian Fluid Mechanics}

\aicescovertitle{Fully-implicit log-conformation formulation of constitutive laws}
\aicescoverauthor{Philipp Knechtges \and Marek Behr \and Stefanie Elgeti}

\begin{document}
\aicescoverpage
\begin{frontmatter}

\title{Fully-implicit log-conformation formulation of constitutive laws}
\author{Philipp Knechtges\corref{cor1}}
\ead{knechtges@cats.rwth-aachen.de}
\author{Marek Behr\corref{}}
\ead{behr@cats.rwth-aachen.de}
\author{Stefanie Elgeti\corref{}}
\ead{elgeti@cats.rwth-aachen.de}

\cortext[cor1]{Corresponding author}

\address{Chair for Computational Analysis of Technical Systems (CATS), RWTH Aachen University, 52056 Aachen, Germany\\
Center for Computational Engineering Science (CCES), RWTH Aachen University, 52056 Aachen, Germany}

\begin{abstract}
Subject of this paper is the derivation of a new constitutive law in terms of the logarithm of the conformation tensor
that can be used as a full substitute for the 2D governing equations of the Oldroyd-B, Giesekus and other models.
One of the key features of these new equations is that -- in contrast to the original log-conf equations given 
by Fattal and Kupferman (2004) --
these constitutive equations combined
with the Navier-Stokes equations constitute a self-contained, non-iterative system of partial differential equations.
In addition to its potential as a fruitful source for understanding the mathematical subtleties of the models from a new perspective,
this analytical description also allows us to fully utilize the Newton-Raphson algorithm in numerical simulations,
which by design should lead to reduced computational effort.
By means of the confined cylinder benchmark we will show that a finite element discretization of these new equations delivers
results of comparable accuracy to known methods.
\end{abstract}

\end{frontmatter}

\section{Introduction}
Viscoelastic phenomena are important for a variety of industrial and medical applications, as, e.g., plastics profile extrusion
and the design of blood pumps. Regardless of the application, the numerical simulation of flows of viscoelastic fluids
often leads to difficulties, when the Weissenberg number,
which relates the elastic forces to the viscous effects, is increased. This challenge has become known as the
High Weissenberg Number Problem \cite{Owens2002}. The difficulty is enhanced by the fact that it has so far not been sufficiently
clarified whether the lack in simulation accuracy should be attributed to purely numerical deficiencies or is an inherent trait
of the utilized constitutive models. One of the more recent approaches to resolve the former are the so-called
log-conformation – or shortly log-conf – formulations going back to \cite{Fattal2004}.

The log-conf formulations are applicable to models of the form
\begin{align}
	\partial_t \bm{\sigma} + (\bm{u}\cdot \nabla) \bm{\sigma} - (\nabla \bm{u})\bm{\sigma}
			- \bm{\sigma}(\nabla\bm{u})^T
		=& - \frac{1}{\lambda} P(\bm{\sigma})\, ,
	\label{eqn:gen_conf_form}
\end{align}
where $\bm{u}$ is a $d$-dimensional velocity vector, $\bm{\sigma}$ the conformation tensor, $\lambda$ the relaxation time and
$P(\bm{\sigma})$ an analytic function. Examples are the Oldroyd-B model \cite{Oldroyd1950} with $P(\bm{\sigma}) = \bm{\sigma} -1$ and the Giesekus model
\cite{Giesekus1982} with ${P(\bm{\sigma}) = \bm{\sigma} -1 + \alpha (\bm{\sigma}-1)^2}$ and the mobility factor $\alpha \in[0,1]$.
It has been shown in \cite{Hulsen1990} that these models require that $\bm{\sigma}$
maintains positive-definiteness through time if the initial data is also positive-definite.
A violation of this condition through the numerical algorithm has been observed to lead to unrecoverable failure of the simulation.
The idea of the log-conf approach is to inherently respect this condition by replacing the original primal degrees of freedom,
i.e., the conformation tensor $\bm{\sigma}$ or the polymeric stress $\bm{T}$, by a new field $\bm{\Psi}$ that is related to the
conformation tensor by the matrix exponential function $\bm{\sigma} = \exp(\bm{\Psi})$; hence the name log-conformation formulation.

Of all possibilities, the choice of the exponential function as a means of assuring positive-definiteness can
be fortified when considering the properties of Lie groups, which are manifolds with a group structure. An important
class of Lie groups are the matrix groups, like the general linear group $\mbox{GL}(d,\mathbb{R})$, which consists of all invertible $d\times d$ matrices.
In the constitutive equation, the tensorial degrees of freedom, like $\bm{\sigma}$, are part of a submanifold of $\mbox{GL}(d,\mathbb{R})$,
which is constituted by the symmetric positive-definite matrices. This space is different as compared to the spaces containing the vectorial degrees of freedom,
which are their own tangent space. The latter is not the case for general manifolds, as for example the symmetric positive-definite matrices.
Nonetheless, the notion of the tangent space is important, since coordinate advancements within the tangent space of a manifold
are guaranteed to remain within the manifold — an advantage when numerically advancing the coordinates.
Fortunately, as $\mbox{GL}(d,\mathbb{R})$ is a Lie group
the matrix exponential function maps the tangent space of the identity element -- also known as the Lie algebra
$\mathfrak{gl}(d,\mathbb{R})=\mathbb{R}^{d\times d}$ --
to the corresponding connected component of the Lie group. Furthermore, the subspace of the symmetric matrices of $\mathfrak{gl}(d,\mathbb{R})$ is mapped
onto the symmetric positive-definite matrices, such that this particular subspace is the natural choice for a vector space for $\bm{\Psi}$.
It should not be left unmentioned that one can still consider other functions than the matrix exponential function to ensure
positive-definiteness, as is, e.g., done in \cite{Balci2011} by the quadratic function.

Apart from the choice of a suitable transforming function, the more intricate task is the derivation of a replacement for the
original constitutive equation that is formulated in terms of the new degrees of freedom. Several approaches have so far been described
\cite{Fattal2004,Coronado2007}. In \cite{Coronado2007}, $\bm{\sigma}$ is replaced by $\exp{\bm{\Psi}}$ in the original
constitutive equation in order to obtain the new equation. Although appealing at first sight, this approach advects $\exp{\bm{\Psi}}$ instead of $\bm{\Psi}$,
leading to possible difficulties in the stabilization of the resulting numerical discretization \cite{Kane2009}.
\cite{Fattal2004} derives the new constitutive equation based on a decomposition of the velocity gradient $\nabla \bm{u}$.
This decomposition leads to an equation with an intrinsically iterative character.
In this paper we derive a new constitutive equation that has neither of these shortcomings.
One of its key features is that it can be stated in a closed form together with the Navier-Stokes equations.
The full derivation has so far been performed for two space dimensions, whereas the three-dimensional case is still subject to current research.

The procedure is outlined in the following fashion. The derivation of the new constitutive equation will be performed in Section \ref{sec:logconf} with the help
of several lemmata, which can be found in \ref{sec:exponentialmapping}. Section \ref{sec:numimpl} introduces the numerical implementation
of this new method, which is subsequently verified by means of the well-known confined cylinder benchmark in Section \ref{sec:benchmark}.
The results are compared to those in \cite{Hulsen2005,Claus2013,Fan1999}.

\section{Log-Conformation}\label{sec:logconf}
For further calculations we will introduce the strain tensor
\begin{align*}
	\varepsilon(\bm{u}) = \frac{1}{2}\left(\nabla\bm{u} + \nabla\bm{u}^T\right)\, ,
\end{align*}
as well as the vorticity tensor
\begin{align*}
	\Omega(\bm{u}) =& \frac{1}{2}\left(\nabla\bm{u} - \nabla\bm{u}^T\right)\, ,
\end{align*}
such that we can rewrite Eq.~\eqref{eqn:gen_conf_form} as
\begin{align}
	\partial_t \bm{\sigma} + (\bm{u}\cdot \nabla) \bm{\sigma} - (\varepsilon(\bm{u})+\Omega(\bm{u}))\bm{\sigma}
			- \bm{\sigma}(\varepsilon(\bm{u})-\Omega(\bm{u}))
		=& - \frac{1}{\lambda} P(\bm{\sigma})\,.
	\label{eqn:gen_conf_form2}
\end{align}
In this section we are going to show that if $\bm{\Psi}$ satisfies
\begin{align}
	\label{eqn:logconf}
	\partial_t\bm{\Psi} + (\bm{u}\cdot \nabla) \bm{\Psi} + [\bm{\Psi},\Omega(\bm{u})]
			- 2 \sum_{n=0}^\infty \frac{B_{2n}}{(2n)!} \{\bm{\Psi},\varepsilon(\bm{u})\}_{2n} &=
			- \frac{1}{\lambda} P\left(e^{\bm{\Psi}}\right) e^{-\bm{\Psi}}\, ,
\end{align}
then $\bm{\sigma}=\exp\bm{\Psi}$ satisfies the original constituitive equation \eqref{eqn:gen_conf_form2}.
In Eq.~\eqref{eqn:logconf}, $B_i$ denote the Bernoulli numbers, $[\bm{X},\bm{Y}]=\bm{X}\bm{Y}-\bm{Y}\bm{X}$ the usual commutator
and $\{\bm{X},\bm{Y}\}_n$ the iterated commutator, which is defined as
\begin{align*}
	\{\bm{X},\bm{Y}\}_n =& [\bm{X},\{\bm{X},\bm{Y}\}_{n-1}]\\
	\{\bm{X},\bm{Y}\}_0 =& \bm{Y}\, .
\end{align*}

Before we come to the proof we will first discuss some properties and prerequisites of this equation.

\begin{rem}[Sobolev spaces and Banach algebras]
The analysis of partial differential equations (PDEs) is highly entangled with the theory of Sobolev spaces.
Therefore, we will assume that $\bm{\Psi}$ is contained in a Sobolev space.
The first thing one realizes when looking at $\bm{\sigma}=\exp\bm{\Psi}$ is that one needs to make sense of the
exponential mapping, which should also map, if possible, into the same Sobolev space.
Mathematically speaking we need a Sobolev space that becomes, equipped with the
pointwise matrix multiplication, a Banach algebra, such that we can define
an analytical functional calculus (cf. \cite[Theorem 10.27]{Rudin}). Restricting ourselves for the moment to the stationary problem
and assuming that $\bm{\Psi}\in H^n(\mathbb{R}^d, \mathbb{R}^{\frac{d(d+1)}{2}})$ it turns out to be
sufficient to demand $n > d/2$ to make the components of $\bm{\Psi}$ lie within a
Banach algebra \cite[Theorem 4.39]{Adams2003}. $\bm{\sigma}$, as well as $P(\bm{\sigma})$, would then also be contained
in $H^n(\mathbb{R}^d, \mathbb{R}^{\frac{d(d+1)}{2}})$.

Moving to the time-dependent setting, we are going to introduce the spaces
\begin{align}
\label{eqn:defH}
\begin{split}
	\mathcal{H} &= C^1([0,T],H^{s-1}(\Omega))\cap C^0([0,T],H^s(\Omega))\\
	\mathcal{H}'&= C^0([0,T],H^{s-1}(\Omega))\, ,
\end{split}
\end{align}
with $s>d/2$ and $\Omega$ being a Lipschitz-bounded domain.
Here, the fact that the multiplications $H^{s-1}(\Omega)\times H^s(\Omega) \to H^{s-1}(\Omega)$ and
$H^s(\Omega)\times H^s(\Omega)\to H^s(\Omega)$ are continuous
\cite[Corollary \S 1.1.1]{giraultraviart86} lets us conclude that $\mathcal{H}$ denotes a Banach algebra. Furthermore,
this multiplication can be extended to a continuous multiplication $\cdot : \mathcal{H}' \times \mathcal{H} \to \mathcal{H}'$. Now deriving the
Banach algebra $H = \mathcal{H}^{d\times d}$ and Banach space $H' = \mathcal{H}'^{d\times d}$, as well as
symmetrized variants thereof
\begin{align*}
	H_{sym} &= \{\bm{X}\in H | \bm{X}^T = \bm{X}\}\\
	H_{sym}' &= \{\bm{X}\in H' | \bm{X}^T = \bm{X}\}\, ,
\end{align*}
we are going to search for solutions of Eq.~\eqref{eqn:logconf} in $H_{sym}$. The space $H'$ will serve as the Banach space
containing the derivatives, since from $\bm{\Psi}\in H_{sym}$ it follows that $\partial_t\bm{\Psi},\nabla\bm{\Psi}\in H_{sym}'$. Moreover, requiring
$\varepsilon(\bm{u})\in H_{sym}'$ lets us interpret all summands in Eq.~\eqref{eqn:logconf} as elements of $H'$.

Allowing to formulate the theory in a Sobolev space setting is, from the theoretical point of view, one of the key advantages of our method
compared to the original log-conf formulation \cite{Fattal2004}, although one has to add that it is not restricted to the choice in \eqref{eqn:defH} and
there are other spaces that fulfill the requirements on $\mathcal{H}$ and $\mathcal{H}'$, fully listed in \ref{sec:exponentialmapping}.
Examples are the smooth function spaces, in which all equations can be thought of as pointwise
evaluations of the specific degrees of freedom. The latter is especially helpful for comprehension
since most of the following proofs are purely algebraic in their nature.

\end{rem}

\begin{rem}[Well-definedness of the series]\label{rem:bernoullinumbers}
We have already outlined in the last paragraph that all summands of the series are elements of $H'$.
What is left to consider is the absolute convergence of the series.
It can be analyzed using the generating function definition of the Bernoulli numbers. Together with $B_1 = -\frac{1}{2}$
as the only non-zero odd Bernoulli number it can be stated as
\begin{align}\label{eqn:genform_bernoulli}
	\sum_{n=0}^\infty \frac{B_{2n}}{(2n)!} x^{2n} &= \frac{x}{2} + \frac{x}{e^x-1}\quad \forall |x|< 2\pi\,.
\end{align}
Furthermore, the inequality $||[\bm{\Psi},\varepsilon(\bm{u})]_{2n}||_{H'} \leq 2^{2n} ||\bm{\Psi}||^{2n}_{H}||\varepsilon(\bm{u})||_{H'}$ and the fact
that the Bernoulli numbers are alternating ($(-1)^{n+1}B_{2n} > 0$ if $n\geq1$)  guarantee that formula \eqref{eqn:logconf} is well-defined
at least for $||\bm{\Psi}||_H<\pi$.
Later we will alleviate this condition for the two-dimensional case.
\end{rem}

\begin{rem}[Symmetry]
As the only two terms containing derivatives of $\bm{\Psi}$, namely $\partial_t\bm{\Psi}$ and $(\bm{u}\cdot \nabla) \bm{\Psi}$,
are clearly symmetric matrices, one also wants the other terms of the formula to be symmetric,
since otherwise one would unnecessarily constrain the number of degrees of freedom by a pure algebraic identity. Although not strictly forbidden,
one could in this case argue that the model would not reflect the "natural" degrees of freedom of the underlying physical nature.
Furthermore, a more practical concern is that it would limit the admissible choices for the boundary conditions.
Fortunately, this is not the case:
One can assert for commutators that if $\bm{X}$ is symmetric and $\bm{Y}$ is
antisymmetric then $[\bm{X},\bm{Y}]$ has to be symmetric. This argument directly applies to the term involving $\Omega(\bm{u})$ which is by definition
the antisymmetric part of the strain tensor. By the same argument one then also sees that $\{\bm{\Psi},\varepsilon(\bm{u})\}_{2n}$ has to be symmetric.
So all terms involved in Eq.~\eqref{eqn:logconf} can be understood as symmetric matrices.
\end{rem}

After having defined the setting we have everything at hand to prove the theorem that encompasses Eq.~\eqref{eqn:logconf}.
\begin{thm}\label{thm:firstmainthm}
Given $\bm{u}\in C^0([0,T],H^s(\Omega,\mathbb{R}^d))$, let $\bm{\Psi}\in H_{sym}$ with $||\bm{\Psi}||_H<\pi$ satisfy Eq.~\eqref{eqn:logconf}, then
$\bm{\sigma}=\exp\bm{\Psi}\in H_{sym}$ solves the original constituitive equation \eqref{eqn:gen_conf_form2}.
\end{thm}
\begin{proof}
In a first step, we apply Eq.~\eqref{eqn:derivative1} to the advective-derivative of the conformation tensor
\begin{align*}
	\left(\partial_t + \bm{u}\cdot\nabla\right) \bm{\sigma} =& \sum_{k=0}^\infty \frac{1}{(k+1)!} \{\bm{\Psi},
			\left(\partial_t + \bm{u}\cdot\nabla\right) \bm{\Psi}\}_{k}\, \bm{\sigma}
\end{align*}
where we now will insert Eq.~\eqref{eqn:logconf}
\begin{align}
\begin{split}
	\left(\partial_t + \bm{u}\cdot\nabla\right) \bm{\sigma} =& - \frac{1}{\lambda} P(\bm{\sigma})
			- \sum_{k=0}^\infty \frac{1}{(k+1)!} \{\bm{\Psi},\Omega(\bm{u})\}_{k+1}\, \bm{\sigma}\\
		& + 2 \sum_{k=0}^\infty \frac{1}{(k+1)!} \sum_{n=0}^\infty \frac{B_{2n}}{(2n)!} \{\bm{\Psi},\varepsilon(\bm{u})\}_{2n+k}\, \bm{\sigma}\, .
\end{split}
\label{eqn:theorem_intermediate}
\end{align}
Here, the fact that $P(e^{\bm{\Psi}})e^{-\bm{\Psi}}$ commutes with $\bm{\Psi}$ has been already incorporated. The second summand
can be evaluated using Lemma \ref{lem:hadamard} and a simple index shift
\begin{align*}
	\sum_{k=0}^\infty \frac{1}{(k+1)!} \{\bm{\Psi},\Omega(\bm{u})\}_{k+1}\, \bm{\sigma} =&
			\sum_{k=0}^\infty \frac{1}{k!} \{\bm{\Psi},\Omega(\bm{u})\}_{k}\, \bm{\sigma} - \Omega(\bm{u})\,\bm{\sigma}\\
		=& \bm{\sigma}\,\Omega(\bm{u})- \Omega(\bm{u})\,\bm{\sigma}\\
		=& [\bm{\sigma},\Omega(\bm{u})]\, .
\end{align*}
The third term is processed by augmenting the series with the odd Bernoulli numbers, of which only $B_1=-\frac{1}{2}$ is non-zero,
and then rearranging the series, such that powers of $\bm{\Psi}$ are collected
\begin{align*}
	\sum_{k=0}^\infty \frac{1}{(k+1)!} \sum_{n=0}^\infty \frac{B_{2n}}{(2n)!} \{\bm{\Psi},\varepsilon(\bm{u})\}_{2n+k}\, \bm{\sigma}
		=& \sum_{k=0}^\infty \frac{1}{(k+1)!} \sum_{n=0}^\infty \frac{B_{n}}{n!} \{\bm{\Psi},\varepsilon(\bm{u})\}_{n+k}\, \bm{\sigma}\\
			& - B_1 \sum_{k=0}^\infty \frac{1}{(k+1)!} \{\bm{\Psi},\varepsilon(\bm{u})\}_{k+1}\, \bm{\sigma}\\
		=& \sum_{i=0}^\infty \{\bm{\Psi},\varepsilon(\bm{u})\}_{i} \, \bm{\sigma} \sum_{n=0}^i \frac{B_{n}}{n!(i-n+1)!}\\
			& +\frac{1}{2} \sum_{k=0}^\infty \frac{1}{(k+1)!} \{\bm{\Psi},\varepsilon(\bm{u})\}_{k+1}\, \bm{\sigma}\, .
\end{align*}
Now we use the recursive definition of the Bernoulli numbers
\begin{align*}
	\sum_{n=0}^i \frac{B_{n}}{n!(i-n+1)!} =& \left\{\begin{array}{ll} 1 & \mbox{for}\, i=0\\ 0 & \mbox{for}\, i\geq 1\end{array}\right.\,
\end{align*}
which can be derived from the generating function definition by comparing coefficients of the left and right hand side
of $1=\left(\sum_k \frac{B_k}{k!}x^k\right)\left(\frac{e^x-1}{x}\right)$.
This together with another application of Lemma \ref{lem:hadamard} finally yields
\begin{align}\label{eqn:firstmainthm}
	\sum_{k=0}^\infty \frac{1}{(k+1)!} \sum_{n=0}^\infty \frac{B_{2n}}{(2n)!} \{\bm{\Psi},\varepsilon(\bm{u})\}_{2n+k}\, \bm{\sigma}
		=& \frac{1}{2} \varepsilon(\bm{u})\bm{\sigma} + \frac{1}{2} \bm{\sigma}\varepsilon(\bm{u})\, .
\end{align}
Pulling all ends together, Eq.~\eqref{eqn:theorem_intermediate} amounts to
\begin{align*}
	\left(\partial_t + \bm{u}\cdot\nabla\right) \bm{\sigma} =& - \frac{1}{\lambda} P(\bm{\sigma})
			- [\bm{\sigma},\Omega(\bm{u})] + \varepsilon(\bm{u})\bm{\sigma} + \bm{\sigma}\varepsilon(\bm{u})\, ,
\end{align*}
which had to be proven.
\end{proof}

As already mentioned in Remark \ref{rem:bernoullinumbers}, it is not really satisfactory to have the bound $||\bm{\Psi}||_H < \pi$, which is necessary
to guarantee absolute convergence of the series. In the following theorem, we will show how to dissolve this bound by identifying a recursion relation
for the iterated commutator as the one given in Lemma \ref{lem:2Dsimp}. It can be used to replace the series by an analytical function. Unfortunately,
one cannot state a single recursion relation as in Lemma \ref{lem:2Dsimp} for arbitrary dimensionality $d$, but has to restrict oneself to a specific $d$.
In the following, we will carry out the details for two dimensions.
The three-dimensional case is far more elaborate and therefore still subject to our current research.
\newpage
\subsection{2D case}

\begin{thm}
Let the velocity field $u\in C^0([0,T],H^s(\Omega,\mathbb{R}^2))$ be given. If $\bm{\Psi}\in H_{sym}$ is a solution of
\begin{align}
\begin{split}
	&\partial_t\bm{\Psi} + (\bm{u}\cdot \nabla) \bm{\Psi} + [\bm{\Psi},\Omega(\bm{u})]
			+ \frac{1}{\lambda} P\left(e^{\bm{\Psi}}\right) e^{-\bm{\Psi}} - 2\varepsilon(\bm{u})\\ &\quad
			- 2 \left(\begin{array}{cc} -\Psi_{12} & \gamma(\bm{\Psi}) \\ \gamma(\bm{\Psi}) & \Psi_{12}\end{array}\right)
			\left[\gamma(\bm{\Psi})\varepsilon(\bm{u})_{12} - \Psi_{12}\gamma(\varepsilon(\bm{u}))\right] \cdot f(\bm{\Psi}) =
			0\, ,
\end{split}
\label{eqn:2DPsi}
\end{align}
with
\begin{align*}
	f(\bm{\Psi}) =& \frac{1}{\gamma(\bm{\Psi})^2+\Psi_{12}^2} \left(
			\sqrt{\gamma(\bm{\Psi})^2+\Psi_{12}^2} + \frac{2\sqrt{\gamma(\bm{\Psi})^2+\Psi_{12}^2}}{\exp\left(2\sqrt{\gamma(\bm{\Psi})^2+\Psi_{12}^2}\right)-1}
			- 1 \right)
\end{align*}
and $\gamma(\bm{C})=\frac{1}{2}\left(C_{11}-C_{22}\right)$, then
the conformation tensor $\bm{\sigma} = \exp\bm{\Psi}\in H_{sym}$ solves the original constitutive equation~\eqref{eqn:gen_conf_form2}.
\end{thm}
\begin{proof}
The proof is twofold: in a first step we will show that the assertion is true for $||\bm{\Psi}||_H<\pi$ and then in a second step that
this restriction is only artificial.

As one can already guess from the comparison of Eq.~\eqref{eqn:logconf} and Eq.~\eqref{eqn:2DPsi} one needs to
replace the series in Eq.~\eqref{eqn:logconf} by an analytical function, which then together with
Theorem \ref{thm:firstmainthm} already yields the conclusion for $||\bm{\Psi}||_H<\pi$.
For that we will first split off the $n=0$ term from the series. Applying Lemma \ref{lem:2Dsimp}
with $\bm{A}=\bm{\Psi}$ and $\bm{B} = \varepsilon(\bm{u})$ and collecting the $n$-dependent terms, we just have to evaluate
\begin{align*}
	f(\bm{\Psi}) =& \sum_{n=1}^\infty \frac{B_{2n}}{(2n)!} 2^{2n} (\gamma(\bm{\Psi})^2 + \Psi_{12}^2)^{n-1}\, .
\end{align*}
The generating function definition of the even Bernoulli numbers (Eq.~\eqref{eqn:genform_bernoulli}) gives us then, after splitting off the $n=0$ term,
the final form of $f(\bm{\Psi})$.

The proof of the second part is in principle similar to that of Theorem \ref{thm:firstmainthm},
just with the difference that the Wilcox Lemma in its initial form is used:
\begin{align*}
	\left(\partial_t + \bm{u}\cdot\nabla\right) \bm{\sigma} =& \int_0^1 e^{(1-\alpha)\bm{\Psi}}
			\left(\left(\partial_t + \bm{u}\cdot\nabla\right) \bm{\Psi}\right) e^{\alpha\bm{\Psi}}\, d\alpha\, .
\end{align*}
Plugging in Eq.~\eqref{eqn:2DPsi} we know for most of the terms the result due to Corollary \ref{cor:wilcox}, as we have already shown in Theorem \ref{thm:firstmainthm}. Only the terms including $\varepsilon(\bm{u})$ need to be reconsidered, since
the series involved in Eq.~\eqref{eqn:logconf} is the only reason for the bound $||\bm{\Psi}||_H<\pi$.
Let us introduce a variable $\beta$ and prove the more generic result
\begin{align}
	\int_0^1 e^{(1-\alpha)\beta\bm{\Psi}} A_\beta e^{\alpha\beta\bm{\Psi}}\, d\alpha
		=& \frac{1}{2} \varepsilon(\bm{u}) e^{\beta\bm{\Psi}} + \frac{1}{2} e^{\beta\bm{\Psi}} \varepsilon(\bm{u})
\label{eqn:analyticcontformula1}
\end{align}
with
\begin{align*}
	A_\beta =& \varepsilon(\bm{u}) + \left(\begin{array}{cc} -\Psi_{12} & \gamma(\bm{\Psi}) \\ \gamma(\bm{\Psi}) & \Psi_{12}\end{array}\right)
			\left[\gamma(\bm{\Psi})\varepsilon(\bm{u})_{12} - \Psi_{12}\gamma(\varepsilon(\bm{u}))\right] \cdot f_\beta(\bm{\Psi})
\end{align*}
and
\begin{align*}
	f_\beta(\bm{\Psi}) =& \frac{1}{\gamma(\bm{\Psi})^2+\Psi_{12}^2} \left(\beta
			\sqrt{\gamma(\bm{\Psi})^2+\Psi_{12}^2} + \frac{2\beta\sqrt{\gamma(\bm{\Psi})^2+\Psi_{12}^2}}
				{\exp\left(2\beta\sqrt{\gamma(\bm{\Psi})^2+\Psi_{12}^2}\right)-1}
			- 1 \right)\, .
\end{align*}
The important part to notice is now that the proof of
Theorem \ref{thm:firstmainthm} already implies Eq.~\eqref{eqn:analyticcontformula1} for $|\beta| < \frac{\pi}{||\bm{\Psi}||_H}$, since
\begin{align*}
	\int_0^1 e^{(1-\alpha)\beta\bm{\Psi}} A_\beta e^{\alpha\beta\bm{\Psi}}\, d\alpha
		&\overset{\aligneqn{\eqref{eqn:hadamard-int1}}{}}{=} \int_0^1 e^{(1-\alpha)\beta\bm{\Psi}} \left( \sum_{n=0}^\infty \frac{B_{2n}}{(2n)!}
					\{\beta \bm{\Psi}, \varepsilon(\bm{u})\}_{2n}\right)
					e^{\alpha\beta\bm{\Psi}}\, d\alpha\\
		&\overset{\aligneqn{\eqref{eqn:hadamard-int1}}{\eqref{eqn:hadamard-int1}}}{=} \sum_{k=0}^\infty \frac{1}{(k+1)!} \sum_{n=0}^\infty \frac{B_{2n}}{(2n)!}
					\{\beta \bm{\Psi}, \varepsilon(\bm{u})\}_{2n+k} e^{\beta\bm{\Psi}}\\
		&\overset{\aligneqn{\eqref{eqn:hadamard-int1}}{\eqref{eqn:firstmainthm}}}{=} \frac{1}{2} \varepsilon(\bm{u}) e^{\beta\bm{\Psi}}
					+ \frac{1}{2} e^{\beta\bm{\Psi}} \varepsilon(\bm{u})\, .
\end{align*}
Furthermore, the integrand of the left side of Eq.~\eqref{eqn:analyticcontformula1} is a holomorphic function in a neighborhood of the real axis, which then
-- together with a combination of Cauchy's and Fubini's Theorem -- shows that the whole integral is holomorphic in that region.
As the right-hand side is clearly holomorphic for all $\beta \in \mathbb{C}$, it follows with the uniqueness of analytic continuation that Eq.~ \eqref{eqn:analyticcontformula1} also has to hold for $\beta =1$, which had to be proven originally.
\end{proof}

Note that this proof carries through irrespective of whether we interpret $\bm{\Psi},\varepsilon(\bm{u})$ as matrices or as elements of the Banach space $H'$.
The concept of analytic continuation works in both cases (cf. \cite[Theorem 3.31]{Rudin} for the notion of holomorphy in the
Banach space setting).

\begin{rem}
It shall be noted that $2\sqrt{\gamma(\bm{\Psi})^2+\Psi_{12}^2}$ is -- assuming sufficient regularity --
just the difference between the two eigenvalues of $\bm{\Psi}$ at a given point $x$.
This becomes evident if one looks at the diagonalization of $\bm{\Psi} (x) = \bm{O} \mbox{diag}(\lambda_1,\lambda_2)\bm{O}^T$, which -- together
with the well-known fact that the identity matrix commutes with every other matrix -- yields
\begin{align*}
	\{\bm{\Psi}, \varepsilon(\bm{u})\}_{2n} =& \bm{O} \{\mbox{diag}(\lambda_1,\lambda_2), \bm{O}^T \varepsilon(\bm{u}) \bm{O}\}_{2n} \bm{O}^T\\
		=& \bm{O} \{\mbox{diag}(\lambda_1-\lambda_2,0), \bm{O}^T \varepsilon(\bm{u}) \bm{O}\}_{2n} \bm{O}^T\, .
\end{align*}
Therefore, the $n$-dependent part that is encapsulated in the function $f(\bm{\Psi})$ can only depend on $2\sqrt{\gamma(\bm{\Psi})^2+\Psi_{12}^2}$.
\end{rem}

This last remark indicates also why it is more difficult to find a similar result in 3D: The problem of finding a recursion relation
for the iterated commutator is highly intertwined with the existence of closed analytical expressions for the eigenvalues of a symmetric matrix.

\subsection{Weak form}
The PDE system we are going to consider, consists of the new constitutive equation \eqref{eqn:2DPsi} encompassed by the Navier-Stokes equations
\begin{gather*}
	\nabla\cdot \bm{u} = 0\\
	\rho (\partial_t + \bm{u}\cdot\nabla) \bm{u} + \nabla p - 2\, \mu_S \nabla\cdot\varepsilon(\bm{u}) - \nabla \cdot \bm{T} = 0\, ,
\end{gather*}
where $\bm{T} = \frac{\mu_P}{\lambda} \left(e^{\bm{\Psi}} -1\right)$ denotes the polymeric stress and $\mu_P, \mu_S$
the polymer and solvent viscosity respectively.
In order to state the weak form of this PDE system we introduce the following spaces for velocity and pressure
\begin{align*}
	V =& C^1([0,T],H^{s-1}(\Omega, \mathbb{R}^2))\cap C^0([0,T],H^s_0(\Omega, \mathbb{R}^2))\\
	Q =& C^0 ([0,T], H^{s-1}(\Omega)\cap L^2_{\int = 0}(\Omega))\, .
\end{align*}
Assuming homogeneous Dirichlet boundary conditions for simplicity, the weak form then reads: \textit{Find a solution
$(\bm{u},p,\bm{\Psi})\in V\times Q\times H_{sym}$ such that for $t\in [0,T]$ and all}
\begin{gather*}
	\bm{v}\in \, H^1_0(\Omega,\mathbb{R}^2)\, ,\\
	q \in \, L^2(\Omega)\, ,\\
	\bm{\Phi} \in L^2(\Omega,\mathbb{R}^{2\times 2})
\end{gather*}
\textit{the following equation is fulfilled}
\begin{align}
\label{eqn:weakform2D}
\begin{split}
	0 =& \quad \rho \left(\bm{v},\partial_t \bm{u} + (\bm{u}\cdot \nabla) \bm{u}\right)_\Omega
				+ \frac{\mu_P}{\lambda} \left(\varepsilon(\bm{v}), e^{\bm{\Psi}} - 1\right)_\Omega
				+ 2\mu_S \left(\varepsilon(\bm{v}),\varepsilon(\bm{u})\right)_\Omega  - \left(\nabla\cdot \bm{v}, p\right)_\Omega\\
			& + \left(q, \nabla \cdot \bm{u}\right)_\Omega \\
			& + \frac{\mu_P}{2\lambda} \left(\bm{\Phi}, \partial_t \bm{\Psi} + (\bm{u}\cdot \nabla) \bm{\Psi}
				+  [\bm{\Psi},\Omega(\bm{u})]
				+ \frac{1}{\lambda} P\left(e^{\bm{\Psi}}\right) e^{-\bm{\Psi}} - 2\varepsilon(\bm{u})\right)_\Omega\\
			& - \frac{\mu_P}{\lambda} \left(\bm{\Phi}, \left(\begin{array}{cc} -\Psi_{12} & \gamma(\bm{\Psi}) \\ \gamma(\bm{\Psi}) & \Psi_{12}\end{array}\right)
			\left[\gamma(\bm{\Psi})\varepsilon(\bm{u})_{12} - \Psi_{12}\gamma(\varepsilon(\bm{u}))\right]  \cdot f(\bm{\Psi})\right)_\Omega\, .
\end{split}
\end{align}
The $L^2$-inner products $(\cdot,\cdot)_\Omega$ are defined by
$(\bm{\Phi},\bm{\Psi})_\Omega = \int_\Omega \bm{\Phi} : \bm{\Psi} =\int_\Omega \mbox{tr} (\bm{\Phi}^T\cdot\bm{\Psi})$
for the tensorial fields and in the usual fashion for scalars and vectors.

The choice of premultiplying the constitutive equation with the factor $\frac{\mu_P}{2\lambda}$ is mainly driven by the consideration that if
$\bm{v},q,\bm{\Phi}$ have the same physical dimension as $\bm{u},p,\bm{\Psi}$, then the residual of the weak form has the physical dimension of power.

\section{Numerical implementation}\label{sec:numimpl}

\subsection{Discretization}
We are going to use \textit{equal-order isoparametric space-time Lagrangian finite elements} in combination with a
\textit{GLS/SUPG stabilization in space} and \textit{DG in time}.

\begin{figure}[t]
\footnotesize
\centering
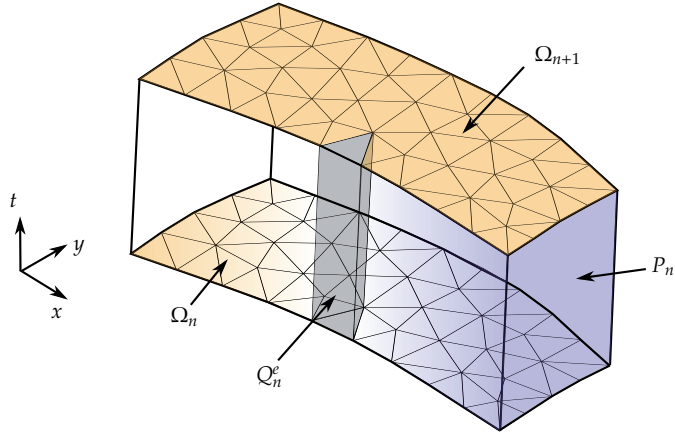
\caption{Illustration of a space-time slab $Q_n$.}
\label{fig:spacetimeslab}
\end{figure}

By \textit{space-time} approach we mean that considering a division of the time span $[0,T]$ into $N$ intervals
$0=t_0 < t_1< \ldots < t_{N+1} = T$, we will also apply
the same slicing to our manifold $Q$ that describes our computational domain in space-time.\footnote{In this paper we will only use $Q= \Omega \times [0,T]$,
but in general this discretization is also suitable for deforming domains.}
The resulting chunks are then called space-time slabs and denoted by
$\{Q_n\}_{0\leq n \leq N-1}$. These slabs are enclosed in time direction by the two boundaries $\Omega_{n} = \Omega_{t_n}$ and
$\Omega_{n+1} = \Omega_{t_{n+1}}$. The trajectory of the spatial boundary through time is defined as
$P_n = \bigcup_{t\in [t_n,t_{n+1}]} \{t\}\times \Gamma_t$, where
$\Gamma_t$ is the spatial boundary of our domain at a given time $t$.

The constructed space-time slabs $Q_n$ will serve as a basis for our triangulation $\mathcal{T}_{h,n}$, which is given by
$\overline{Q_n} = \bigcup_e Q^e_n$. More precisely, the geometrical
basis of our finite elements is derived from a common reference element, which in our case is a space-time prism $\tilde{Q}$. The function that maps $\tilde{Q}$
onto $Q^e_n$ is called $T_{Q^e_n}$. Employing the \textit{isoparametric principle} and the fact that we use $\mathbb{P}_2$ interpolation in space, as well as
$\mathbb{P}_1$ interpolation in time, we can write the mapping as
\begin{align*}
	T_{Q^e_n}(\bm{\xi}) =& \sum_{i=1}^{n_{en}/2} \left(\begin{array}{c} \bm{x}^e_i \\ t_n\end{array}\right) \phi_i (\bm{\xi})
			+ \sum_{i=n_{en}/2 +1}^{n_{en}} \left(\begin{array}{c} \bm{x}^e_i \\ t_{n+1}\end{array}\right) \phi_i (\bm{\xi})\, ,
\end{align*}
where $n_{en}=\dim \mathbb{P}_2\cdot \dim \mathbb{P}_1$ and $\{\phi_i\}$ is the \textit{Lagrange} basis of $\mathbb{P}_2\otimes \mathbb{P}_1$.
In 2D this amounts to $n_{en} = 12$ nodes per element.

Based on this geometry we can start to construct the interpolation space for our degrees of freedom first on a single space-time slab $Q_n$
\begin{align*}
	V_{h,n} =& \left\{v\in C^0(\overline{Q_n}) \mathrel{}\middle|\mathrel{} \forall Q^e_n \in \mathcal{T}_{h,n} , v\circ T_{Q^e_n} \in \mathbb{P}_2\otimes \mathbb{P}_1\right\}\, ,
\end{align*}
and then on the whole space-time manifold by concatenation
\begin{align*}
	V_{h} =& \left\{ v \in L^2(Q) \mathrel{}\middle|\mathrel{} v|_{[t_n,t_{n+1}]} \in V_{h,n}\right\}\, .
\end{align*}
Note that the interpolation functions are continuous in space, but discontinuous in time.

In order to formulate a well-posed discretized problem we need to restrict our test and trial function spaces to subspaces of $V_{h,n}$. Therefore, let $P_{n,\bm{u}}$
denote the part of the space-time boundary $P_n$ that corresponds to a Dirichlet-boundary condition of the velocity $\bm{u}$, whereas $P_{n,\bm{\Psi}}$
corresponds to the $\bm{\Psi}$-Dirchlet boundary.\footnote{Generalizations, where only a few components of $\bm{u}$
or $\bm{\Psi}$ are prescribed, are obvious.}
The trial function space $\mathcal{S}_{h,n}$ and the test function space $\mathcal{V}_{h,n}$ are then given by
\begin{align*}
	\mathcal{S}_{h,n} =& \left\{ (\bm{u},p,\bm{\Psi}) \in (V_{h,n})^d \times V_{h,n} \times (V_{h,n})^{d\cdot(d+1)/2} \mathrel{}\middle|\mathrel{}
				\bm{u}|_{P_{n,\bm{u}}} = \bm{g}_{\bm{u}}, \bm{\Psi}|_{P_{n,\bm{\Psi}}} = \bm{g}_{\bm{\Psi}}\right\}\\
	\mathcal{V}_{h,n} =& \left\{ (\bm{v},q,\bm{\Phi}) \in (V_{h,n})^d \times V_{h,n} \times (V_{h,n})^{d\cdot(d+1)/2} \mathrel{}\middle|\mathrel{}
				\bm{v}|_{P_{n,\bm{u}}} = \bm{0}, \bm{\Phi}|_{P_{n,\bm{\Psi}}} = \bm{0}\right\}\, .
\end{align*}
Furthermore, we denote with $\mathcal{S}_{h}$ the concatenation of the $\mathcal{S}_{h,n}$ spaces.
The fact that the same space $V_{h,n}$ is used  as a basis for the interpolation of all degrees of freedom is usually referred to as \textit{equal-order interpolation}.

In the following, we will formulate a weak problem on each space-time slab.

\subsection{2D case}
Using the terminology of the last section, the aim of this section is to state a stabilized discrete version of the weak form in Eq.~\eqref{eqn:weakform2D}.
We will first state the weak form and then discuss certain aspects about it:

\textit{Starting with $(\bm{u}^h)_0^- = \bm{u}_0$ and
$(\bm{\Psi}^h)_0^- = \bm{\Psi}_0$ we are seeking $(\bm{u}^h,p^h,\bm{\Psi}^h) \in \mathcal{S}_{h}$
such that for all $n\in\{0,1,\ldots,N-1\}$ and all $(\bm{v}^h,q^h,\bm{\Phi}^h)\in \mathcal{V}_{h,n}$ the following equation is fulfilled}
\begin{align}
\label{eqn:weakform2D-disc}
\begin{split}
	0 =& \quad \int_{Q_n} \bm{v}^h \cdot \rho \left(\partial_t \bm{u}^h + (\bm{u}^h \cdot \nabla) \bm{u}^h\right)
					+ \int_{Q_n} \frac{\mu_P}{\lambda} \varepsilon(\bm{v}^h) : \left(e^{\bm{\Psi}^h} -1\right)\\&
				+ \int_{Q_n} 2\mu_s \varepsilon(\bm{v}^h) : \varepsilon(\bm{u}^h) - \int_{Q_n} (\nabla \cdot \bm{v}^h) \cdot p^h
					+ \int_{\Omega_n} (\bm{v}^h)_n^+ \cdot \rho \left((\bm{u}^h)_n^+ - (\bm{u}^h)_n^-\right) \\&
				+ \sum_e \int_{Q_n^e} \tau_{mom} \frac{1}{\rho} \left(\rho (\bm{u}^h \cdot \nabla) \bm{v}^h + \nabla q^h + \mu_S \Delta \bm{v}^h
						- \frac{\mu_P}{\lambda} \nabla \cdot \bm{\Phi}^h\right) \\&\qquad
							\cdot \left(\rho (\partial_t \bm{u}^h + (\bm{u}^h \cdot \nabla) \bm{u}^h)
						+ \nabla p^h - \mu_S \Delta \bm{u}^h - \frac{\mu_P}{\lambda} \nabla \cdot \left(e^{\bm{\Psi}^h}-1\right)\right)\\&
				+ \int_{Q_n} q^h \, (\nabla\cdot \bm{u}^h) \\&
				+ \int_{Q_n} \frac{\mu_P}{2\lambda} \left(\bm{\Phi}^h + \tau_{cons} (\bm{u}^h \cdot\nabla) \bm{\Phi}^h\right) \\&\qquad
							: \left(\partial_t \bm{\Psi}^h + (\bm{u}^h\cdot \nabla) \bm{\Psi}^h + [\bm{\Psi}^h,\Omega(\bm{u}^h)]
								+ \frac{1}{\lambda} P\left(e^{\bm{\Psi}^h}\right) e^{-\bm{\Psi}^h} - 2\varepsilon(\bm{u}^h)\right)\\&
				- \int_{Q_n} \frac{\mu_P}{\lambda} \left(\bm{\Phi}^h + \tau_{cons} (\bm{u}^h \cdot\nabla) \bm{\Phi}^h\right) \\&\qquad
							: \left(\left(\begin{array}{cc} -\Psi_{12}^h & \gamma(\bm{\Psi}^h) \\ \gamma(\bm{\Psi}^h) & \Psi_{12}^h\end{array}\right)
									\left[\gamma(\bm{\Psi}^h)\varepsilon(\bm{u}^h)_{12} - \Psi_{12}^h\gamma(\varepsilon(\bm{u}^h))\right]  \cdot f(\bm{\Psi}^h)\right)\\&
				+ \int_{\Omega_n} (\bm{\Phi}^h)_n^+ : \frac{\mu_P}{2\lambda} \left((\bm{\Psi}^h)_n^+ - (\bm{\Psi}^h)_n^-\right)\, .
\end{split}
\end{align}
Here, the terms describing discontinuities across space-time slabs are defined by
\begin{align*}
	(\bm{u}^h)_n^\pm =& \lim_{\xi\to 0} \bm{u}^h (t_n \pm \xi)\, .
\end{align*}

The following remarks shall be made about the given weak form:
\begin{itemize}
\item The used stabilization is a mixture of an adjoint\footnote{Also known as the Douglas-Wang method \cite{Douglas1989}.}
Galerkin/Least-Squares (GLS) stabilization \cite{Franca1992a,Franca1992b,Behr1993}
for the momentum equation and a plain Streamline Upwind/Petrov-Galerkin (SUPG) stabilization \cite{Brooks1982} for the
constitutive equation. The stabilization of the continuity equation is omitted, since in practical applications polymer flows usually have small Reynolds numbers.
The corresponding element-specific stabilization parameters depend on the element length $h$, $\Delta t = t_{n+1} -t_n$ and a characteristic velocity
$\bm{u}$ evaluated at the element center:
\begin{align*}
	\tau_{mom} =& \mbox{min}\left(\rho \frac{h^2}{314\, \mu}, \frac{h}{2|\bm{u}|}, \frac{\Delta t}{2}\right)\, ,\\
	\tau_{cons} =& \mbox{min}\left(\left(2\frac{|\bm{u}|}{h} + \lambda^{-1}\right)^{-1}, \frac{\Delta t}{2}\right)\, .
\end{align*}
The main motivation for choosing $\nabla \cdot \bm{\Phi}$ instead of $\nabla \cdot e^{\bm{\Phi}}$ in the GLS term is the intention for the given weak form
to satisfy a discrete version of so-called free energy estimates. The latter has been shown in \cite{Boyaval2009} to be essential to prove the
existence of discrete global-in-time solutions for homogeneous boundary conditions.

\item Currently, the implementation only supports the Oldroyd-B model, leading to the substitution
\begin{align*}
	P(e^{\bm{\Psi}^h})e^{-\bm{\Psi}^h} = 1 - e^{-\bm{\Psi}^h}\, .
\end{align*}

\item For the evaluation of the matrix exponential functions we use a combination of the Padé approximants $R_{6,6}$ and a scaling/squaring approach
(cf. \cite{Moler2003}).
The latter means, when trying to calculate $e^{\bm{X}}$ for $\bm{X}\in\mathbb{R}^{d\times d}$,
we first choose $j\in \mathbb{N}$ large enough such that $||\bm{X}||_\infty < 2^j$.
In a second step we compute $\bm{A} = R_{6,6} (\bm{X}\cdot 2^{-j})$ as an approximation of $\exp(\bm{X}\cdot 2^{-j})$ and then finally perform
$j$ in-place squarings of the matrix $\bm{A}$.

\item A point that has not been present in the original weak form in Eq.~\eqref{eqn:weakform2D}
is the evaluation of the derivative of the matrix exponential function.
The discretized version in Eq.~\eqref{eqn:weakform2D-disc} now includes ${\nabla \cdot \left(e^{\bm{\Psi}^h}-1\right)}$ in the GLS term,
which will be dealt with with the help of Corollary \ref{cor:expderv2D}
\begin{align*}
	\nabla \cdot \left(e^{\bm{\Psi}^h}-1\right) =& \sum_{i=1}^2 \hat{\bm{e}}_i^T \cdot e^{\bm{\Psi}^h/2} \bigg[ \partial_i \bm{\Psi}^h
					+ \left(\begin{array}{cc} - \Psi^h_{12} & \gamma(\bm{\Psi}^h)\\ \gamma(\bm{\Psi}^h) & \Psi^h_{12}\end{array}\right) \\&
				\qquad\qquad\qquad\cdot \left[\gamma(\bm{\Psi}^h)\partial_i \Psi^h_{12} - \Psi^h_{12} \partial_i \gamma(\bm{\Psi}^h)\right]
					\cdot g(\bm{\Psi}^h)\bigg] e^{\bm{\Psi}^h/2}\, .
\end{align*}
For the definition of the scalar function $g(\bm{\Psi}^h)$ we refer to Eq.~\eqref{eqn:definitiong}.
\end{itemize}

\subsection{Linearization}
So far, the weak form given in Eq.~\eqref{eqn:weakform2D-disc} is still non-linear. In our implementation, we employed the Newton-Raphson method in order to
linearize the problem. Now consider the weak form in Eq.~\eqref{eqn:weakform2D-disc} to be given in the abstract form: we are searching for a
$\bm{z}^h = (\bm{u}^h, p^h, \bm{\Psi}^h)\in\mathcal{S}_{h}$ such that for each time step $n \in \{0,\ldots, N-1\}$ and all
$\bm{w}^h = (\bm{v}^h, q^h, \bm{\Phi}^h)\in\mathcal{V}_{h,n}$ we have
\begin{align*}
	a_n(\bm{w}^h, \bm{z}^h) =& 0\, .
\end{align*}
Applying Newton's algorithm to each time step separately then reads: starting from some initial guess $\bm{z}^h_{n,0}$ we are searching for
$\delta\bm{z}^h_{n,i} = (\delta \bm{u}^h_{n,i}, \delta p^h_{n,i}, \delta\bm{\Psi}^h_{n,i}) \in \mathcal{V}_{h,n}$ such that
\begin{align}\label{eqn:newtonraphson}
	\left. D a_n (\bm{w}^h, \cdot )\right|_{\bm{z}^h_{n,i}}\, \delta \bm{z}^h_{n,i} =& \, - a_n(\bm{w}^h, \bm{z}^h_{n,i})\quad \forall \bm{w}^h \in \mathcal{V}_{h,n}\, .
\end{align}
The resulting $\delta \bm{z}^h_{n,i}$ is used afterwards to update $\bm{z}^h_{n,i+1} = \bm{z}^h_{n,i} + \delta \bm{z}^h_{n,i}$, which is then
reinserted into the algorithm until the residual of Eq.~\eqref{eqn:weakform2D-disc} becomes small enough.
Here, the directional variational derivative is, as usual, defined as
\begin{align*}
	\left. D a_n (\bm{w}^h, \cdot )\right|_{\bm{z}^h_{n,i}}\, \delta \bm{z}^h_{n,i} =&
			\left.\frac{d}{d\xi}\right|_{\xi=0} a_n(\bm{w}^h, \bm{z}^h_{n,i} + \xi \cdot \delta \bm{z}^h_{n,i})\, .
\end{align*}

For the sake of brevity we will not state the full variational derivative but rather only parts of it. E.g., the material derivative in the momentum equation
becomes
\begin{align*}
	\left. D a_n (\bm{w}^h, \cdot )\right|_{\bm{z}^h_{n,i}}\, \delta \bm{z}^h_{n,i} =& \ldots +
			\int_{Q_n} \bm{v}^h \cdot \rho \left(\partial_t \delta \bm{u}^h_{n,i} + (\delta \bm{u}^h_{n,i} \cdot \nabla) \bm{u}^h_{n,i}
				+ (\bm{u}^h_{n,i} \cdot \nabla) \delta \bm{u}^h_{n,i}\right)+\ldots\, ,
\end{align*}
whereas the velocity DG term is given by
\begin{align*}
	\left. D a_n (\bm{w}^h, \cdot )\right|_{\bm{z}^h_{n,i}}\, \delta \bm{z}^h_{n,i} =& \ldots +
			\int_{\Omega_n} (\bm{v}^h)_n^+ \cdot \rho \left((\delta \bm{u}^h_{n,i})_n^+ - (\bm{u}^h)_n^-\right)+\ldots\, .
\end{align*}
The terms involving $e^{\bm{\Psi}^h}$ are handled using Corollary \ref{cor:expderv2D}. Therefore, the $\bm{\Psi}^h$ contribution
to the momentum equation leads to
\begin{align*}
\begin{split}
	\left. D a_n (\bm{w}^h, \cdot )\right|_{\bm{z}^h_{n,i}}\, \delta \bm{z}^h_{n,i} =& \ldots +
			\int_{Q_n} \frac{\mu_P}{\lambda} \varepsilon(\bm{v}^h) : e^{\bm{\Psi}^h_{n,i}/2} \bigg[\delta \bm{\Psi}^h_{n,i}
				+ \left(\begin{array}{cc} - (\bm{\Psi}^h_{n,i})_{12} & \gamma(\bm{\Psi}^h_{n,i})\\
														\gamma(\bm{\Psi}^h_{n,i}) & (\bm{\Psi}^h_{n,i})_{12}\end{array}\right) \\&
				\qquad\qquad\cdot \left[\gamma(\bm{\Psi}^h_{n,i})(\delta \bm{\Psi}^h_{n,i})_{12}
														- (\bm{\Psi}^h_{n,i})_{12} \gamma(\delta \bm{\Psi}^h_{n,i})\right]
					\cdot g(\bm{\Psi}^h_{n,i})\bigg] e^{\bm{\Psi}^h_{n,i}/2}+\ldots\, .
\end{split}
\end{align*}
It shall also be noted that our implementation does not include all contributions to the variational derivative. More specifically, the derivatives
of the GLS/SUPG-stabilization terms
with respect to the velocity field have been omitted. The reasoning behind this is that these terms are only meant to stabilize the linear equation system.

Further on, by choosing a basis of $\mathcal{V}_{h,n}$ one can reformulate \eqref{eqn:newtonraphson} into a linear equation system in a usual fashion,
which is then accessible to a linear solver like GMRES. In our case, we use an inherently-parallel version of FGMRES \cite{Saad1993}.
The latter is combined with an ILUT preconditioner \cite{Saad1994} that is parallelised using an Additive Schwarz approach with zero overlap.

\section{Confined cylinder benchmark}\label{sec:benchmark}
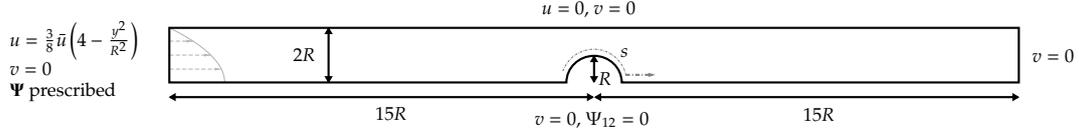
\begin{figure}[t]
\scriptsize
\centering
\definecolor{cb5b5b5}{RGB}{181,181,181}
\definecolor{cb6b6b6}{RGB}{182,182,182}
\definecolor{cb4b4b4}{RGB}{180,180,180}
\definecolor{c808080}{RGB}{128,128,128}

\begin{tikzpicture}[y=0.80pt, x=0.8pt,yscale=-1, inner sep=0pt, outer sep=0pt]
\begin{scope}
  \path[draw=cb5b5b5,dash pattern=on 1.60pt off 0.80pt,line join=miter,line
    cap=butt,miter limit=4.00,line width=0.400pt] (239.9284,953.1596) --
    (264.1242,953.1596);

  \path[draw=cb5b5b5,dash pattern=on 1.60pt off 0.80pt,line join=miter,line
    cap=butt,miter limit=4.00,line width=0.400pt] (239.9284,946.5580) --
    (259.1959,946.5580);

  \path[draw=cb6b6b6,dash pattern=on 1.60pt off 0.80pt,line join=miter,line
    cap=butt,miter limit=4.00,line width=0.400pt] (239.9284,939.9565) --
    (250.8785,939.9565);

  \path[draw=cb4b4b4,line join=round,line cap=butt,line width=0.400pt]
    (239.9284,933.6850) .. controls (240.5234,933.9810) and (241.1082,934.2769) ..
    (241.6827,934.5728) .. controls (242.2572,934.8687) and (242.8214,935.1647) ..
    (243.3753,935.4606) .. controls (243.9293,935.7565) and (244.4730,936.0524) ..
    (245.0065,936.3484) .. controls (245.5399,936.6444) and (246.0631,936.9402) ..
    (246.5761,937.2362) .. controls (247.0890,937.5321) and (247.5916,937.8280) ..
    (248.0841,938.1239) .. controls (248.5765,938.4199) and (249.0586,938.7158) ..
    (249.5305,939.0118) .. controls (250.0024,939.3077) and (250.4641,939.6036) ..
    (250.9154,939.8995) .. controls (251.3668,940.1955) and (251.8079,940.4914) ..
    (252.2388,940.7873) .. controls (252.6697,941.0833) and (253.0903,941.3792) ..
    (253.5006,941.6751) .. controls (253.9110,941.9710) and (254.3110,942.2670) ..
    (254.7009,942.5629) .. controls (255.0907,942.8589) and (255.4703,943.1548) ..
    (255.8396,943.4507) .. controls (256.2089,943.7467) and (256.5679,944.0426) ..
    (256.9168,944.3385) .. controls (257.2656,944.6344) and (257.6041,944.9304) ..
    (257.9324,945.2264) .. controls (258.2606,945.5222) and (258.5787,945.8182) ..
    (258.8864,946.1141) .. controls (259.1942,946.4101) and (259.4917,946.7059) ..
    (259.7789,947.0019) .. controls (260.0662,947.2979) and (260.3431,947.5938) ..
    (260.6099,947.8897) .. controls (260.8766,948.1857) and (261.1330,948.4816) ..
    (261.3793,948.7775) .. controls (261.6255,949.0734) and (261.8614,949.3694) ..
    (262.0871,949.6653) .. controls (262.3128,949.9612) and (262.5282,950.2572) ..
    (262.7334,950.5531) .. controls (262.9386,950.8490) and (263.1335,951.1449) ..
    (263.3182,951.4410) .. controls (263.5028,951.7368) and (263.6772,952.0327) ..
    (263.8413,952.3288) .. controls (264.0055,952.6247) and (264.1593,952.9205) ..
    (264.3030,953.2165) .. controls (264.4466,953.5125) and (264.5800,953.8084) ..
    (264.7031,954.1042) .. controls (264.8262,954.4003) and (264.9390,954.6962) ..
    (265.0416,954.9921) .. controls (265.1442,955.2880) and (265.2365,955.5840) ..
    (265.3186,955.8799) .. controls (265.4007,956.1758) and (265.4725,956.4718) ..
    (265.5340,956.7677) .. controls (265.5956,957.0636) and (265.6469,957.3596) ..
    (265.6879,957.6555) .. controls (265.7290,957.9514) and (265.7597,958.2473) ..
    (265.7803,958.5433) .. controls (265.8008,958.8392) and (265.8111,959.1351) ..
    (265.8111,959.4311);

  \path[draw=black,line width=0.800pt] (239.9284,933.6935) -- (239.9284,959.4395)
    -- (425.7519,959.4395) .. controls (425.7519,952.5122) and (431.5459,946.8966)
    .. (438.6932,946.8966) .. controls (445.8404,946.8966) and (451.6344,952.5122)
    .. (451.6344,959.4395) -- (637.4579,959.4395) -- (637.4579,933.6935) --
    (239.9284,933.6935) -- cycle;

  \path[draw=black,line join=round,line cap=butt,line width=0.800pt]
    (314.4373,959.1010) -- (314.4373,934.0151);

  \path[draw=black,line join=miter,line cap=butt,line width=0.800pt]
    (240.2603,966.3628) -- (438.0295,966.3628);

  \path[fill=black,even odd rule] (316.4283,956.4604) -- (312.4464,956.4604) --
    (314.4373,959.7612) -- cycle;

  \path[fill=black,even odd rule] (312.4464,936.6557) -- (316.4283,936.6557) --
    (314.4373,933.3549) -- cycle;

  \path[fill=black,even odd rule] (242.9149,968.3433) -- (242.9149,964.3823) --
    (239.5966,966.3628) -- cycle;

  \path[fill=black,even odd rule] (435.3749,964.3823) -- (435.3749,968.3433) --
    (438.6932,966.3628) -- cycle;

  \begin{scope}[shift={(199.09654,0)},shift={(0,0)}]
    \path[draw=black,line join=miter,line cap=butt,line width=0.800pt]
      (240.2603,966.3628) -- (438.0295,966.3628);

  \end{scope}
  \begin{scope}[shift={(199.09654,0)},shift={(0,0)}]
    \path[fill=black,even odd rule] (242.9149,968.3433) -- (242.9149,964.3823) --
      (239.5966,966.3628) -- cycle;

  \end{scope}
  \begin{scope}[shift={(199.09654,0)},shift={(0,0)}]
    \path[fill=black,even odd rule] (435.3749,964.3823) -- (435.3749,968.3433) --
      (438.6932,966.3628) -- cycle;

  \end{scope}
  \path[xscale=1.003,yscale=0.997,fill=black] (297.08798,952.37714) node[above
    right,align=left] (text3931) {$2R$};

  \path[xscale=1.003,yscale=0.997,fill=black] (335.30603,980.31116) node[above
    right,align=left] (text3935) {$15R$};

  \path[xscale=1.003,yscale=0.997,fill=black] (534.87665,980.46387) node[above
    right,align=left] (text3939) {$15R$};

  \path[draw=black,line join=miter,line cap=butt,line width=0.800pt]
    (438.6932,959.1010) -- (438.6932,947.2182);

  \path[fill=black,even odd rule] (440.6841,956.4604) -- (436.7022,956.4604) --
    (438.6931,959.7612) -- cycle;

  \path[fill=black,even odd rule] (436.7022,949.8588) -- (440.6841,949.8588) --
    (438.6932,946.5580) -- cycle;

  \path[xscale=1.003,yscale=0.997,fill=black] (439.55783,963.58789) node[above
    right,align=left] (text4012) {$R$};

  \path[fill=cb6b6b6,even odd rule] (247.8923,939.2963) -- (247.8923,940.6166) --
    (251.2106,939.9565) -- cycle;

  \path[fill=cb6b6b6,even odd rule] (256.1880,945.8979) -- (256.1880,947.2182) --
    (259.5063,946.5580) -- cycle;

  \path[fill=cb6b6b6,even odd rule] (261.1654,952.4995) -- (261.1654,953.8198) --
    (264.4837,953.1596) -- cycle;

  \path[draw=c808080,dash pattern=on 1.60pt off 0.80pt on 0.40pt off 0.80pt,line
    join=miter,line cap=butt,miter limit=4.00,line width=0.800pt]
    (453.4724,955.8673) -- (466.0819,955.8673);

  \path[fill=c808080,even odd rule] (464.0909,954.8770) -- (464.0909,956.8575) --
    (467.4092,955.8673) -- cycle;

  \path[cm={{1.93568,0.0,0.0,1.76492,(472.51337,-824.12797)}},draw=c808080,dash
    pattern=on 0.87pt off 0.43pt on 0.22pt off 0.43pt,line join=round,line
    cap=butt,miter limit=4.00,line width=0.433pt]
    (-25.0083,1007.5259)arc(200.000:350.000:7.958984 and 8.887);

  \path[xscale=1.003,yscale=0.997,fill=black] (641.91241,953.04492) node[above
    right,align=left] (text3298) {$v=0$};

  \path[xscale=1.003,yscale=0.997,fill=black] (164.83913,933.51263) node[below
    right,align=left] (text3321) {$u =
    \frac{3}{8}\bar{u}\left(4-\frac{y^2}{R^2}\right)$\\$v = 0$\\$\bm{\Psi}$
    prescribed};

  \path[xscale=1.003,yscale=0.997,fill=black] (409.72345,983.72736) node[above
    right,align=left] (text3130) {$v = 0$, $\Psi_{12}=0$};

  \path[xscale=1.003,yscale=0.997,fill=black] (413.01999,931.44708) node[above
    right,align=left] (text3171) {$u=0$, $v=0$};

  \path[xscale=1.003,yscale=0.997,fill=black] (449.96594,950.20898) node[above
    right,align=left] (text3384) {$s$};

\end{scope}

\end{tikzpicture}
\caption{Sketch of the used geometry including the boundary conditions.}
\label{fig:geom_confinedcyl}
\end{figure}

We are going to benchmark our implementation with the so-called confined cylinder problem, for which a great variety of results is already
accessible in the literature \cite{Liu1998,Sun1999,Fan1999,Hulsen2005,Coronado2007,Afonso2009,Claus2013}. More specifically, in this benchmark problem one considers a cylinder
that is confined between two walls with a ratio of the channel width to cylinder diameter of $2$, as can be seen in Fig. \ref{fig:geom_confinedcyl}.
Furthermore, in order to reduce the numerical workload we restrict ourselves to the symmetric solutions.
Assuming now a steady Poiseuille flow at the inlet, the aim of this section is to measure
several performance quantities of the steady stream flowing around the cylinder, e.g., the drag on the cylinder or the polymeric stress in
the wake of the cylinder.

\subsection{Setup}

The boundary conditions are analogous to the ones found in the appropriate literature. No-slip boundary conditions are applied
at the channel wall and the cylinder surface.
On the centerline, we incorporated the symmetry by a slip boundary condition, meaning that we set $v=0$ and $\Psi_{12}=0$, whereas at
the outflow only $v=0$ was enforced. For the boundary conditions on the inflow, we have chosen to prescribe a fully-developed Poiseuille solution
of the Oldroyd-B model. The latter is well-known for the velocity degrees, but for the log-conf field one has to derive the terms in a three-step approach
by first diagonalizing the known expressions for the conformation tensor, then applying the logarithm on the eigenvalues and at last collapsing the
eigendecomposition again. The result is subsequently given by
\begin{align*}
\begin{split}
\begin{aligned}[c]
	\Psi_{11}^{in} =& \frac{1}{2}\left(p - \frac{q}{\sqrt{1+\left(\lambda \partial_y u\right)^{-2}}}\right)\\
	\Psi_{12}^{in} =& -\frac{1}{2} \frac{q}{o}\\
	\Psi_{22}^{in} =& \frac{1}{2} \left(\frac{p}{\left(\lambda \partial_y u\right)^2} + \frac{q}{o}\right)
\end{aligned}
\quad\mbox{where}\quad
\begin{aligned}[c]
	o =& \sqrt{\left(\lambda \partial_y u\right)^2\cdot\left(1+\left(\lambda \partial_y u\right)^2\right)}\\
	p=& \ln\left(1+\left(\lambda \partial_y u\right)^2\right)\\
	q =& \ln\left(1+2\left(\left(\lambda \partial_y u\right)^2-o\right)\right)\, .
\end{aligned}
\end{split}
\end{align*}

\begin{figure}[t]
\centering
\includegraphics[width=\textwidth]{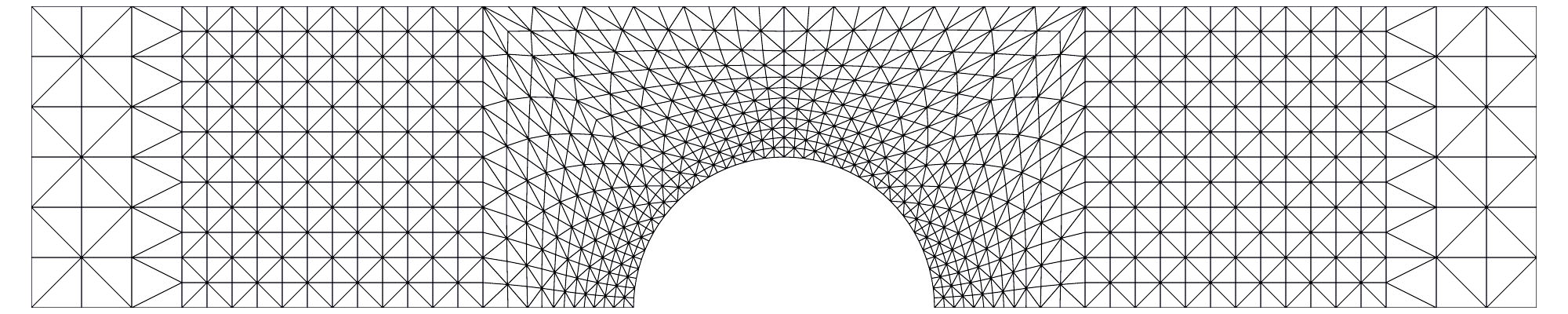}
\caption{Excerpt of the Mesh M1.}
\label{fig:mesh}
\end{figure}

We use structured triangular meshes with equal spacing on the cylinder, as can be seen in Fig. \ref{fig:mesh}. All finer meshes
were obtained by doubling the number of elements on the cylinder and adjusting the surrounding mesh accordingly.
Further mesh properties can be found in Tab. \ref{tab:meshproperties}.

\begin{table}[ht]
\footnotesize
\centering
\begin{tabular}{lccccc}
\toprule
 & M1 & M2 & M3 & M4 & M5 \\
\midrule
Number of elements on the half-cylinder & 48 & 96 & 192 & 384 & 768 \\
Total number of nodes & 5353 & 20785 & 81889 & 326595 & 1298307 \\
Total number of elements & 2532 & 10104 & 40368 & 162144 & 646848 \\
Krylov-space dimension & 200 & 200 & 200 & 200 & 400 \\
ILUT maximal fill-in & 200 & 200 & 200 & 200 & 200 \\
ILUT threshold & $10^{-4}$ & $10^{-4}$ & $10^{-4}$ & $10^{-4}$ & $10^{-4}$ \\
Number of cores & 16 & 32 & 64 & 128 & 256 \\
\bottomrule
\end{tabular}
\caption{Mesh and solver properties.}
\label{tab:meshproperties}
\end{table}

As in most of the literature, we also examine the creeping flow limit with vanishing Reynolds number, which in our case was enforced
by omitting the advective terms in the momentum equation part of Eq.~\eqref{eqn:weakform2D-disc},
as well as rendering the SUPG-term $\frac{h}{2|\bm{u}|}$ in the stabilization parameter $\tau_{mom}$ ineffective.
The steady state equation is implemented in a similar fashion: The corresponding terms in the main equation and the stabilization are neglected.
The latter distinguishes us from part of the literature, where steady-state simulations are not applied, but rather instationary simulations are
conducted until the quantities
of interest have settled to a constant value \cite{Hulsen2005,Claus2013}. Despite being superior to the instationary approach in terms of simulation time,
the stationary approach puts more pressure on the Newton-Raphson solver, which has to be alleviated by a consecutive ramping up of the Weissenberg number.
For completeness, we should mention that the Weissenberg number is in our notation defined as
\begin{align*}
	Wi =& \frac{\lambda \bar{u}}{R}\, ,
\end{align*}
where $\bar{u}$ denotes the average inflow velocity and $R$ the cylinder radius.
Furthermore, as in the literature, we use a viscosity ratio of $\beta=\mu_S/\mu = 0.59$ for the benchmark.

\subsection{Results}

Basis of the comparison is the computation of the drag on the cylinder for different Weissenberg numbers.
For better comparability with existing results we introduce the dimensionless drag coefficient
\begin{align*}
	K &= \frac{2}{\mu \bar{u}} \int_{\Gamma_{HC}} \hat{\bm{e}}_x^T \left[-p^h + 2 \mu_S \varepsilon (\bm{u}^h)
			+ \frac{\mu_P}{\lambda}\left(e^{\bm{\Psi}^h}-1\right)\right] \bm{n}\, ,
\end{align*}
where $\Gamma_{HC}$ is the one-dimensional manifold describing the half-cylinder surface and $\bm{n}$ the corresponding
unit normal.

\begin{table}[h]
\footnotesize
\centering
\begin{tabular}{ccccccccc}
\toprule
\multirow{2}[3]{*}{$Wi$} & \multicolumn{8}{c}{$K$}\\
\cmidrule(lr){2-9}
 & M1 & M2 & M3 & M4 & M5 & \cite{Hulsen2005} & \cite{Claus2013} & \cite{Fan1999} \\
\midrule
0.1 & 130.3706 & 130.3613 & 130.3620 & 130.3625 & 130.3626 & 130.363 & 130.364 & 130.36 \\
0.2 & 126.6609 & 126.6288 & 126.6254 & 126.6252 & 126.6252 & 126.626 & 126.626 & 126.62 \\
0.3 & 123.2622 & 123.2008 & 123.1922 & 123.1913 & 123.1912 & 123.193 & 123.192 & 123.19 \\
0.4 & 120.6953 & 120.6080 & 120.5931 & 120.5914 & 120.5912 & 120.596 & 120.593 & 120.59 \\
0.5 & 118.9615 & 118.8505 & 118.8291 & 118.8263 & 118.8260 & 118.836 & 118.826 & 118.83 \\
0.6 & 117.9542 & 117.8048 & 117.7798 & 117.7756 & 117.7752 & 117.775 & 117.776 & 117.78 \\
0.7 & 117.5430 & 117.3416 & 117.3193 & 117.3155 & 117.3157 & 117.315 & 117.316 & 117.32 \\
0.75& 117.5108 & 117.2940 & 117.2747 & 117.2733 & 117.2752 &  &  &  \\
0.8 & 117.5639 & 117.3539 & 117.3365 & 117.3395 & 117.3454 & 117.373 & 117.368 & 117.36 \\
0.85& 117.6809 & 117.5116 & 117.4925 & 117.5016 & 117.5138 &  &  &  \\
0.88& 117.7743 & 117.6495 & 117.6265 & 117.6402 & 117.6567 &  &  &  \\
0.89& 117.8085 & 117.7022 & 117.6774 & 117.6927 & 117.7107 &  &  &  \\
0.9 & 117.8442 & 117.7584 & 117.7312 & 117.7483 & 117.7678 & 117.787 & 117.812 & 117.80 \\
\bottomrule
\end{tabular}
\caption{Results for the drag coefficient $K$ compared to results from literature. The values from literature are always the finest mesh results. In the case of
\cite{Fan1999}, the MIX0 results are utilized.}
\label{tab:drag}
\end{table}

\begin{figure}[htb]
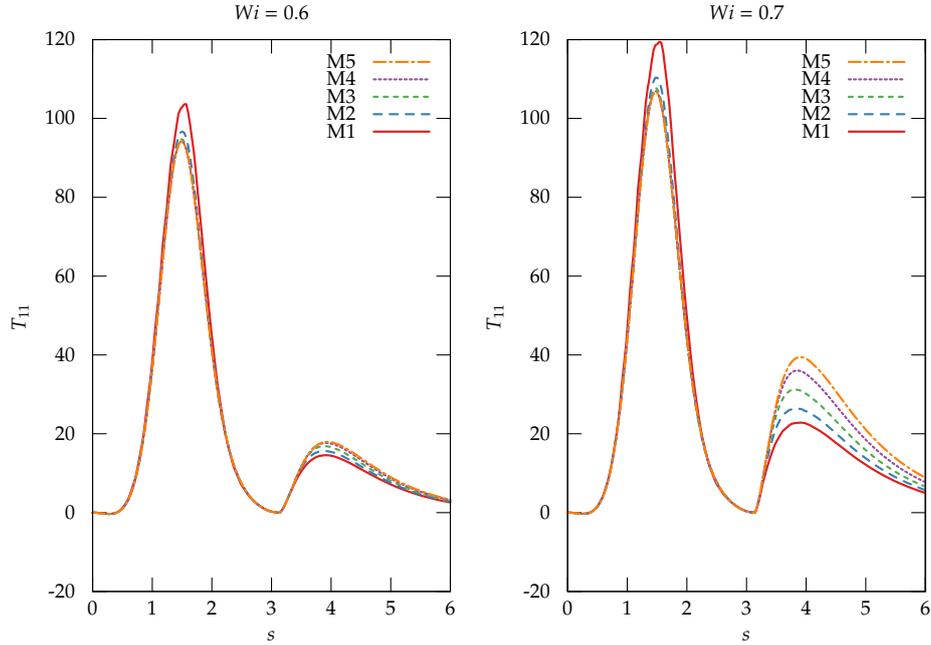

\centering
\include{Fig/T11WakePlot}
\caption{The $T_{11}$ component of the polymeric stress along and in the wake of the cylinder for different Weissenberg numbers.}
\label{fig:T11WakePlot}
\end{figure}

\begin{figure}[htb]
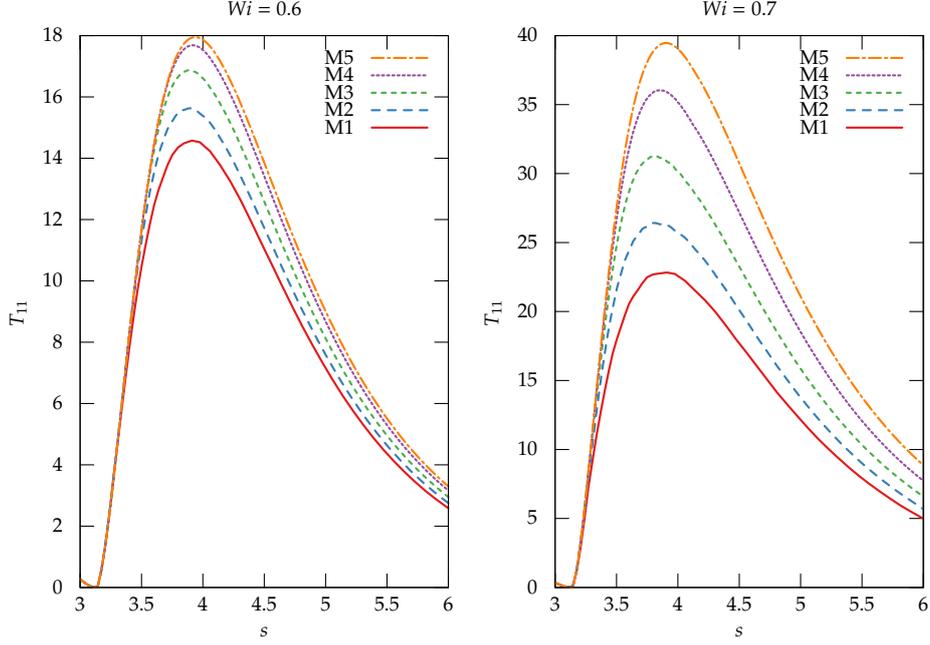

\centering
\include{Fig/T11WakePlotDetail}
\caption{Detail plot of $T_{11}$ in the wake of the cylinder.}
\label{fig:T11WakePlotDetail}
\end{figure}

Tab. \ref{tab:drag} reveals that up to $Wi\leq 0.7$, our results agree quite well with the existing results in literature.
Above $Wi=0.7$ the divergence of the results increases across the different publications.

Another point that becomes directly apparent while looking at Tab. \ref{tab:drag} is that the step sizes between two consecutive $Wi$ calculations
had to be reduced with increasing Weissenberg number: starting with $\Delta Wi = 0.1$ and ending with $\Delta Wi = 0.01$.
The underlying difficulties, that seem to be symptomatic for large jumps in the Weissenberg number, manifested themselves most of the time in the lack
of convergence of the linear solver in the second or third Newton-Raphson step. Hence, one gains the impression
that the first Newton-Raphson step drives the solution into a direction where the following linear system is ill-conditioned. It shall also be noted that
this problem becomes more severe with increasing mesh size, although it is not clear what this is to be attributed to: On the one hand
there is an inherent increase in the condition number due to mesh refinement and on the other hand one
can partially alleviate the problem by increasing the Krylov-space dimension of the GMRES.
This interplay of two competing effects makes it particularly difficult to determine a limiting Weissenberg
number for a sufficiently fine mesh.

Furthermore, looking at Fig. \ref{fig:T11WakePlot} and \ref{fig:T11WakePlotDetail}
one also notices that we -- in accordance with the literature -- cannot claim to have reached mesh convergence
of the polymeric stress in the wake of the cylinder for Weissenberg numbers greater than $0.6$. As had already been concluded in previous investigations of
the benchmark, this raises concerns about whether at higher Weissenberg numbers the simulation results are still physical.

\section{Conclusion and discussion}
In this paper we have proposed a new constitutive equation of log-conf type that can be used as a drop-in replacement for a variety of existing
constitutive models,
e.g., the Oldroyd-B model or the Giesekus model. In contrast to the existing work of Fattal and Kupferman \cite{Fattal2004} we do not need to introduce
an a-priori iterative procedure which applies an eigenvalue-type decomposition to the strain tensor, but rather obtain in combination with the Navier-Stokes
equations a self-contained fully-implicit system of PDEs. Especially this knowledge of the analytic structure of the constitutive equation is what allows us
to then attain a fast Newton-Raphson algorithm in our numerical implementation of the model.
First numerical tests have shown that the performance of this new method is at least comparable to
the existing log-conf methods, but, at least in the case of the investigated Oldroyd-B model, also suffers from the same weaknesses
as the original log-conf method;
namely breakdown of the simulations already at low Weissenberg numbers.
In that regard we share the opinion of the authors of \cite{Hulsen2005}, that this might be due to the Oldroyd-B model allowing infinte extension of the
polymer under finite elongation rates. Further investigations have to be conducted to see whether other constitutive models
reduce these problems and how our method compares to other implementations.

In addition to being a useful tool for numerical simulations of constitutive equations, we also hope that our new formulation will be fruitful for the
discussion of the Weissenberg problem in general, since in addition to the numerical analysis it might give a new perspective on the problem
from the purely analytical point of view. Since this is a rather intricate topic on its own, it is out of scope for this paper.

\section{Acknowledgements}
The authors gratefully acknowledge support from the German Research Foundation (DFG) grant
”Computation of Die Swell Behind a Complex Profile Extrusion Die Using a Stabilized Finite Element Method for Various Thermoplastic Polymers”
and the DFG program GSC 111 (AICES Graduate School).
The computations were conducted on computing clusters provided by the Jülich Aachen Research Alliance (JARA). In addition to that the authors also
want to thank the On-Line Encyclopedia of Integer Sequences (OEIS) \cite{OEIS} for being a helpful resource in identifying the Bernoulli numbers as coefficients
in Eq.~\eqref{eqn:logconf}.
Special thanks goes to Prof. Arnold Reusken for commenting on the early versions of the manuscript.

\appendix

\section{Properties of the matrix exponential mapping}\label{sec:exponentialmapping}
\subsection{General considerations}
There are several lemmas originating from the Lie group and Banach algebra theory that help us derive the
log-conformation formulation. In the following section we will require
\begin{itemize}
\item a commutative Banach algebra $\mathcal{H}$,
\item a Banach space $\mathcal{H}'$,
\item a set of continuous differential operators $\partial_i: \mathcal{H}\to \mathcal{H}'$,
\item a continuous embedding $\mathcal{H} \subseteq \mathcal{H}'$,
\item and that the multiplication on $\mathcal{H}$ shall be extensible to a continuous multiplication $\cdot : \mathcal{H}' \times \mathcal{H} \to \mathcal{H}'$.
\end{itemize}
From this setting we will derive another Banach algebra $H = \mathcal{H}^{d\times d}$ and Banach space $H' = \mathcal{H}'^{d\times d}$, as well as
symmetrized variants thereof
\begin{align*}
	H_{sym} &= \{\bm{X}\in H | \bm{X}^T = \bm{X}\}\\
	H_{sym}' &= \{\bm{X}\in H' | \bm{X}^T = \bm{X}\}\, .
\end{align*}
The space $H_{sym}$ will for example serve us as the space containing $\bm{\Psi}$, and $H_{sym}'$ as the space in which the constitutive
equation is formulated.

\begin{lem}[Hadamard]
Let $\bm{X}$ be an element of $H$ and $\bm{Y}$ an element of $H'$, then the following identity holds
\begin{align*}
	e^{\bm{X}}\bm{Y}e^{-\bm{X}} =& \bm{Y} + [\bm{X},\bm{Y}] + \frac{1}{2!}[\bm{X},[\bm{X},\bm{Y}]] + \ldots\\
		=& \sum_{n=0}^\infty \frac{1}{n!} \{\bm{X},\bm{Y}\}_{n}\, ,
\end{align*}
where $[\bm{X},\bm{Y}]=\bm{X}\bm{Y}-\bm{Y}\bm{X}$ denotes the usual commutator and we recursively define
$\{\bm{X},\bm{Y}\}_n = [\bm{X},\{\bm{X},\bm{Y}\}_{n-1}]$ with $\{\bm{X},\bm{Y}\}_0 = \bm{Y}$.
\label{lem:hadamard}
\end{lem}
\begin{proof}
At first one has to recognize that $e^{t\bm{X}}\bm{Y}e^{-t\bm{X}}$ is a holomorphic function of $t\in\mathbb{C}$.
Now, as in the case of the matrix algebra, the assertion is a consequence of evaluating the Taylor series of $e^{t\bm{X}}\bm{Y}e^{-t\bm{X}}$
around $t=0$ at $t=1$.
\end{proof}
For later use, a chain-rule type relation for the exponential mapping is required. A first version can be found in the following lemma, of which a slightly different
variant can be traced back to \cite{Karplus1948}.

\begin{lem}[Wilcox]
Let $\bm{X}$ be an element of $H$, then the following identity holds
\begin{align*}
	\partial_i e^{\bm{X}(x)} =& \int_0^1 e^{(1-\alpha)\bm{X}(x)} \left(\partial_i \bm{X}(x)\right) e^{\alpha\bm{X}(x)}\, d\alpha\, .
\end{align*}
\label{lem:wilcox}
\end{lem}
\begin{proof}
We refer to \cite{Wilcox1966} for the details of the proof.
\end{proof}

An important corollary to these lemmata is
\begin{cor}
Let $\bm{X} \in H$, then
\begin{align}
\begin{split}
	\left(\partial_i e^{\bm{X}(x)}\right)e^{-\bm{X}(x)} =& \partial_i \bm{X} + \frac{1}{2!} [\bm{X},\partial_i \bm{X}] + \ldots\\
		=& \sum_{n=0}^\infty \frac{1}{(n+1)!} \{\bm{X}(x),\partial_i \bm{X}(x)\}_n
\end{split}
\label{eqn:derivative1}
\end{align}
holds, as well as
\begin{align}
	e^{-\bm{X}(x)/2}\left(\partial_i e^{\bm{X}(x)}\right)e^{-\bm{X}(x)/2} =& \sum_{n=0}^\infty \frac{1}{(2n+1)!}
			\frac{1}{2^{2n}} \{\bm{X}(x),\partial_i \bm{X}(x)\}_{2n}\, .
\label{eqn:derivative2}
\end{align}
\label{cor:wilcox}
\end{cor}
\begin{proof}
We will restrict ourselves to the proof of the second equality since the first one is similar. Lemma \ref{lem:hadamard} and \ref{lem:wilcox} combined give
\begin{align*}
	e^{-\bm{X}(x)/2}\left(\partial_i e^{\bm{X}(x)}\right)e^{-\bm{X}(x)/2} =&
			\int_0^1 e^{(\frac{1}{2}-\alpha)\bm{X}(x)} \left(\partial_i \bm{X}(x)\right) e^{-(\frac{1}{2}-\alpha)\bm{X}(x)}\, d\alpha\\
		=& \sum_{j=0}^\infty \frac{1}{j!} \{\bm{X}(x),\partial_i \bm{X}(x)\}_j \int_0^1 \left(\frac{1}{2}-\alpha\right)^{j}\, d\alpha
\end{align*}
The integral is clearly zero for odd $j$ and for even $j=2n$ we obtain
\begin{align*}
	\int_0^1 \left(\frac{1}{2}-\alpha\right)^{2n}\, d\alpha =& \int_{-1/2}^{1/2} x^{2n}\, dx = \frac{1}{2n+1}\cdot \frac{1}{2^{2n}}\, .
\end{align*}
Hence, using $n$ as a summation index yields the desired result.
\end{proof}

\begin{rem}
In the last proof one can also substitute $\partial_i \bm{X}$ by an arbitrary $\bm{Y} \in H'$ and see that the following identities hold
\begin{align}
	\int_0^1 e^{(1-\alpha)\bm{X}} \bm{Y} e^{\alpha\bm{X}}\, d\alpha
		=& \sum_{n=0}^\infty \frac{1}{(n+1)!} \{\bm{X},\bm{Y}\}_n e^{\bm{X}} \label{eqn:hadamard-int1}\\
		=& e^{\bm{X}/2} \sum_{n=0}^\infty \frac{1}{(2n+1)!} \frac{1}{2^{2n}} \{\bm{X},\bm{Y}\}_{2n} e^{\bm{X}/2}\, . \label{eqn:hadamard-int2}
\end{align}
\end{rem}

\subsection{2D case}
In this section we will show how to substitute the series in Corollary \ref{cor:wilcox} by analytical functions for elements of $H_{sym}$ in the case of $d=2$.
\begin{lem}\label{lem:2Dsimp}
For $\bm{A}\in H_{sym},\bm{B}\in H_{sym}'$ and $d=2,n\geq 1$, we can show
\begin{align}\label{eqn:2Dsimp}
	\{\bm{A},\bm{B}\}_{2n} =& 2^{2n} \left(\begin{array}{cc} -A_{12} & \gamma(\bm{A}) \\ \gamma(\bm{A}) & A_{12}\end{array}\right)
			\left[\gamma(\bm{A})B_{12} - A_{12}\gamma(\bm{B})\right] \left(\gamma(\bm{A})^2+A_{12}^2\right)^{n-1}\, ,
\end{align}
where
\begin{align*}
	\gamma(\bm{C}) =& \frac{1}{2} (C_{11} - C_{22})\quad \forall \bm{C}\in H\, .
\end{align*}
\end{lem}
\begin{proof}
Without loss of generality, we can assume $\bm{A},\bm{B}$ being traceless, which amounts to
\begin{align*}
	\bm{A} =& \left(\begin{array}{cc} \gamma(\bm{A}) & A_{12} \\ A_{12} & - \gamma(\bm{A})\end{array}\right)
\end{align*}
and $\bm{B}$ given analogously.
We will now prove the formula using induction, starting with $n = 1$, in which case an algebraic calculation yields
\begin{align*}
	\{\bm{A},\bm{B}\}_{2} =& 4 \left(\begin{array}{cc} - A_{12} & \gamma(\bm{A}) \\ \gamma(\bm{A}) & A_{12}\end{array}\right)
			\left[\gamma(\bm{A})B_{12} - A_{12}\gamma(\bm{B})\right]\, .
\end{align*}
Assuming that Eq.~\eqref{eqn:2Dsimp} holds for $n-1$ we can now reiterate
\begin{align*}
	\{\bm{A},\bm{B}\}_{2n} =&\{\bm{A},\{\bm{A},\bm{B}\}_{2}\}_{2n-2}\\
		=& 2^{2n-2} \left(\begin{array}{cc} - A_{12} & \gamma(\bm{A}) \\ \gamma(\bm{A}) & A_{12}\end{array}\right)
			\left[\gamma(\bm{A})(\{\bm{A},\bm{B}\}_{2})_{12} - A_{12}\gamma(\{\bm{A},\bm{B}\}_{2})\right] \left(\gamma(\bm{A})^2+A_{12}^2\right)^{n-2}\,
\end{align*}
where
\begin{align*}
	\gamma(\{\bm{A},\bm{B}\}_{2}) =& - 4 A_{12} \left[\gamma(\bm{A})B_{12} - A_{12}\gamma(\bm{B})\right]\\
	(\{\bm{A},\bm{B}\}_{2})_{12} =& 4 \gamma(\bm{A}) \left[\gamma(\bm{A})B_{12} - A_{12}\gamma(\bm{B})\right]\, ,
\end{align*}
such that Eq.~\eqref{eqn:2Dsimp} also holds for $n$.
\end{proof}

Combining Corollary \ref{cor:wilcox} and Lemma \ref{lem:2Dsimp} then yields

\begin{cor}\label{cor:expderv2D}
For $\bm{X}\in H_{sym}$ and $d=2$ we can express the derivative of the exponential mapping as
\begin{align*}
	&e^{-\bm{X}(x)/2}\left(\partial_i e^{\bm{X}(x)}\right)e^{-\bm{X}(x)/2}\\&\quad =  \partial_i \bm{X}(x) + \left(\begin{array}{cc} - \bm{X}_{12} & \gamma(\bm{X})\\
				\gamma(\bm{X}) & \bm{X}_{12}\end{array}\right)\left[\gamma(\bm{X})\partial_i \bm{X}_{12} - \bm{X}_{12} \partial_i \gamma(\bm{X})\right]
				\cdot g(\bm{X})\, ,
\end{align*}
with
\begin{align}
	g(\bm{X}) =& \left(\gamma(\bm{X})^2+\bm{X}_{12}^2\right)^{-3/2}\cdot\left(\sinh\left(\sqrt{\gamma(\bm{X})^2+\bm{X}_{12}^2}\right) -
				\sqrt{\gamma(\bm{X})^2+\bm{X}_{12}^2}\right)\, .
	\label{eqn:definitiong}
\end{align}
\end{cor}
\begin{proof}
Inserting the result of Lemma \ref{lem:2Dsimp} in Eq.~\eqref{eqn:derivative2} yields
\begin{align*}
	&e^{-\bm{X}(x)/2}\left(\partial_i e^{\bm{X}(x)}\right)e^{-\bm{X}(x)/2}\\&\quad =  \partial_i \bm{X}(x) + \left(\begin{array}{cc} - \bm{X}_{12} & \gamma(\bm{X})\\
				\gamma(\bm{X}) & \bm{X}_{12}\end{array}\right)\left[\gamma(\bm{X})\partial_i \bm{X}_{12} - \bm{X}_{12} \partial_i \gamma(\bm{X})\right]
				\\&\quad\quad\cdot \sum_{n=1}^\infty \frac{1}{(2n+1)!} \left(\gamma(\bm{X})^2+\bm{X}_{12}^2\right)^{n-1}\, .
\end{align*}
Using the Taylor series $\sinh(x) = \sum_{n=0}^\infty \frac{1}{(2n+1)!} x^{2n+1}$, the assertion of the corollary follows immediately.
\end{proof}

\bibliographystyle{elsarticle-num}
\bibliography{references}


\end{document}